\documentclass[11pt]{article}
\usepackage{amsmath, amsfonts, amsthm, amssymb}
\usepackage{enumerate}

\theoremstyle{plain}
\newtheorem{proposition}{Proposition}[section]
\newtheorem{theorem}[proposition]{Theorem} 
\newtheorem{lemma}[proposition]{Lemma}

\newtheorem{question}[proposition]{Question}
\newtheorem{corollary}[proposition]{Corollary}

\theoremstyle{definition}

\newtheorem{remark}[proposition]{Remark}

\def\sdisc{{\mathrm{sDisc}}}
\def\disc{{\mathrm{Disc}}}
\def\rootdisc{c}
\def\rootsdisc{f}
\def\harmonics{\mathcal{H}}
\def\mat{{\mathcal{M}}}
\def\tzmat{{\mathcal{N}}}
\def\hermat{{\mathrm{Her}}}
\def\symmat{{\mathrm{Sym}}}
\def\ideal{{\mathcal{I}}}
\def\tange{c}
\def\inv{d}
\def\bideg{{\mathrm{bideg}}}
\def\cocoa{{\hbox{\rm C\kern-.13em o\kern-.07em C\kern-.13em o\kern-.15em A}}}

\def\field{{\mathbb{F}}}
\def\mr{{\mathbb{R}}}
\def\mn{{\mathbb{N}}}
\def\mc{{\mathbb{C}}}
\def\mz{{\mathbb{Z}}}

\def\tr{{\mathrm{tr}}}

\def\covariants{\mathrm{Cov}}

\begin{document}

\title{Hermitian matrices with a bounded number of eigenvalues}
\author{M. Domokos\thanks{Partially supported by OTKA  NK81203 and K101515.}
\\ {\small R\'enyi Institute of Mathematics, Hungarian Academy of 
Sciences,} 
\\ {\small 1053 Budapest, Re\'altanoda utca 13-15., Hungary} \\ 
{\small E-mail: domokos.matyas@renyi.mta.hu } 
}
\date{}
\maketitle 
\begin{abstract} 
Conjugation covariants of matrices are applied to study  the real algebraic variety consisting of complex Hermitian matrices with a bounded number of distinct eigenvalues. 
A minimal generating system of the vanishing ideal of degenerate three by three Hermitian matrices is given, and the structure of the corresponding coordinate ring 
as a module over the special unitary group is determined. 
%the Hilbert series of  the corresponding coordinate ring is computed.  
The method applies also for  degenerate real symmetric three by three matrices. 
For arbitrary $n$ partial information on the minimal degree component of the vanishing ideal of the variety of $n\times n$ Hermitian matrices with a bounded number of eigenvalues is obtained, and some known results on sum of squares presentations of subdiscriminants of real symmetric matrices are extended to the case of complex Hermitian matrices.


\end{abstract}

%%%%%%%%%%%%%%%%%%%%%%%%%%%%%%%%%%%

\section{Introduction}\label{sec:intro} 

Let $\field$ be the field $\mr$ of real numbers or the field $\mc$ of complex numbers. For a matrix $A\in \mc^{n\times n}$ denote $\bar A$ and $A^T$  the complex conjugate and  transpose of $A$, respectively. 
Fix a positive integer $n\ge 2$, and let $\mat$ be one of the following $\field$-subspaces of $\mc^{n\times n}$: 
\begin{itemize}
\item[(a)] the Hermitian matrices 
$\hermat(n)=\{A\in \mc^{n\times n}\mid \bar A=A^T\}$ 

\item[(b)] the real symmetric matrices $\symmat(n,\mr)=\{A\in \mr^{n\times n}\mid A^T=A\}$

\item[(c)]  all $n\times n$ complex matrices $\mc^{n\times n}$ 

\item[(d)] the complex symmetric matrices 
$\symmat(n,\mc)=\{A\in\mc^{n\times n}\mid A^T=A\}$
\end{itemize}

For $k=0,1,\dots,n-1$ consider the following  subset of $\mat$: 
\[\mat_k:=\{A\in\mat\mid \deg(m_A)\le n-k\}\] 
where $m_A$ stands for the minimal polynomial of the matrix $A$. 
Clearly $\mat_0=\mat$, $\mat_k\supset\mat_{k+1}$, and for a fixed $k\in\{0,1,\dots,n-1\}$ we have the inclusions 
\[\begin{array}{ccc}(\mc^{n\times n})_k & \supset & \hermat(n)_k \\
\cup& & \cup\\
\symmat(n,\mc)_k  &\supset &\symmat(n,\mr)_k\end{array}\]
We have also the equalities $ \hermat(n)_k=(\mc^{n\times n})_k\cap \hermat(n)$ and  
$\symmat(n,\mr)_k=\symmat(n,\mc)_k\cap\symmat(n,\mr)=\hermat(n)_k\cap \mr^{n\times n}$.  
Obviously $\mat_k$ is the common zero locus in $\mat$ of the coordinate functions of the polynomial map 
\[\mathcal{P}_k:\mat\to \bigwedge^{n-k+1}\mat,\quad A\mapsto I_n\wedge A\wedge A^2\wedge\dots\wedge A^{n-k}\] 
where $I_n$ is the $n\times n$ identity matrix and $\bigwedge^l\mat$ is the $l$th exterior power of $\mat$. 
In particular, $\mat_k$ is an affine algebraic subvariety of the affine space $\mat$, and  
it is  natural to raise the following question: 

\begin{question}\label{question} 
Do the coordinate functions of the polynomial map $\mathcal{P}_k$ 
generate the vanishing ideal 
$\ideal(\mat_k)$ in $\field[\mat]$ of the affine algebraic subvariety $\mat_k\subset \mat$? 
\end{question} 

Above $\field[\mat]$ is the coordinate ring of $\mat$, so $\field=\mr$ in cases (a), (b) whereas $\field=\mc$ in cases (c), (d), and 
$\field[\mat]$ is a polynomial ring over $\field$ in  $\dim_{\field}(\mat)$ variables. 
Recall that  the {\it vanishing ideal} of $\mat_k$ is 
\[\ideal(\mat_k):=\{f\in\field[\mat]\mid f\big\vert_{\mat_k}\equiv 0\}\triangleleft\field[\mat]\] 
We have $\mat_0=\mat$, so $\ideal(\mat_0)$ is the zero ideal,  and  $\mathcal{P}_0$ is the zero map. 
From now on we focus on $\mat_{k+1}$ and $\ideal(\mat_{k+1})$ where $k=0,1,\dots,n-2$. 

Our original interest was in the real cases  (a) and (b):  then $\field=\mr$ and all $A\in\mat$ are diagonalizable with real eigenvalues, hence 
\begin{equation}\label{eq:mat_kreal}
\mat_{k+1}=\{A\in\mat\mid A\mbox{ has at most }n-k-1 \mbox{ distinct eigenvalues}\}
\end{equation}
It follows from \eqref{eq:mat_kreal} that in the real cases $\mat_{k+1}$ (for $k=0,1,\dots,n-2$) is the zero locus of a single polynomial $\sdisc_k\in\mr[\mat]$, 
defined by 
\[\sdisc_k(A):=\sum_{1\le i_1<\dots<i_{n-k}\le n}\prod_{1\le s<t\le n-k}(\lambda_{i_s}-\lambda_{i_t})^2\] 
where $\lambda_1,\dots,\lambda_n$ are the eigenvalues of $A$. Note that $\sdisc_k(A)$ coincides with the {\it $k$-subdiscriminant} of the characteristic polynomial of $A$ 
(we refer to Chapter 4 of  \cite{roy} for basic properties of subdiscriminants), and $\sdisc_k$ is a homogeneous polynomial function on $\mat$ of degree 
$(n-k)(n-k-1)$. In the special case $k=0$ we recover the {\it discriminant} $\disc=\sdisc_0$. 
The ideal $\ideal(\mat_{k+1})$ is generated by homogeneous elements (with respect to the standard grading on the polynomial ring 
$\field[\mat]=\bigoplus_{d=0}^{\infty}
\field[\mat]_d$. 
In \cite{domokos2} it was deduced from the Kleitman-Lov\'asz theorem (cf. Theorem 2.4 in \cite{lovasz}) that 
$\frac 12\deg(\sdisc_k)=\binom{n-k}{2}$ is the minimal degree of a non-zero homogeneous component of 
$\ideal(\mat_{k+1})=\bigoplus_{d=0}^{\infty}\ideal(\mat_k)_d$ (in fact 
\cite{domokos2} deals with the case $\mat=\symmat(n,\mr)$ only, but the proof of Corollary 5.3 in loc. cit. works also for the case 
$\mat=\hermat(n)$, see Proposition~\ref{prop:subdiscunique} and Theorem~\ref{thm:main4} (i) in the present paper). 
Since the polynomial map $\mathcal{P}_{k+1}$ is homogeneous of degree $\binom{n-k}{2}$, 
its coordinate functions are contained in  the  homogeneous component $\ideal(\mat_{k+1})_{\binom{n-k}{2}}$. 
So an affirmative answer to Question~\ref{question} would imply in particular that $\ideal(\mat_{k+1})$ is generated by its minimal degree non-zero homogeneous component.   

In Section~\ref{sec:zariski} we observe that the Zariski closure of $\hermat(n)_k$ in the complex affine space $\mc^{n\times n}$ is $(\mc^{n\times n})_k$, and the Zariski closure of $\symmat(n,\mr)_k$ in the complex affine space $\symmat(n,\mc)$ is $\symmat(n,\mc)_k$, see Proposition~\ref{prop:closure}. 
This implies the following: 

\begin{corollary}\label{cor:idealofclosure} 
Let $\mat$ be $\hermat(n)$ respectively $\symmat(n,\mr)$, and $\mc\otimes_{\mr}\mat$ its complexification $\mc^{n\times n}$ respectively $\symmat(n,\mc)$. 
We have the equality 
\[\ideal((\mc\otimes_{\mr}\mat)_k)=\mc\otimes_{\mr}\ideal(\mat_k) \] 
where we make the standard identification  $\mc[\mc\otimes_{\mr}\mat]=\mc\otimes_{\mr}\mr[\mat]$. 
\end{corollary}

Note  that in the complex cases (c), (d) by the Hilbert Nullstellensatz the coordinate functions of $\mathcal{P}_k$ generate $\ideal(\mat_k)$ up to radical, hence 
by Corollary~\ref{cor:idealofclosure} this holds also in the real cases, giving some evidence for an affirmative answer to Question~\ref{question}:   

\begin{corollary}\label{cor:radical} 
The coordinate functions of $\mathcal{P}_k$ generate $\ideal(\mat_k)$ up to radical.  
\end{corollary} 

The answer to Question~\ref{question} is trivially yes for $k=0$, $n$ arbitrary, and it is straighforward to check that the answer is yes for $k=n-1$, $n$ arbitrary (since $\mat_{n-1}$ 
consists of scalar matrices, so it is a linear subspace of $\mat$, and thus its ideal is generated by linear polynomials). 
The smallest interesting case therefore is $n=3$ and $k=1$. The results of the present paper show in particular that the answer to Question~\ref{question} is yes also in this case, see Corollary~\ref{cor:main2} (i) and (iv).  

%\begin{theorem}\label{thm:main1} 
%For $n=3$  the ideal $\ideal(\mat_1)$ is generated by the coordinate functions of $\mathcal{P}_1:\mat\to\bigwedge^3\mat$, 
%$A\mapsto I_3\wedge A\wedge A^2$. 
%\end{theorem} 

By Corollary~\ref{cor:idealofclosure} it is sufficient to deal with the complex cases (c), (d). Then there is an action of a semisimple complex linear algebraic group $G$ on $\mat$. Namely $G$ is the {\it complex special linear group} $SL(n,\mc)$ in case (c) and the {\it complex special orthogonal group} $SO(n,\mc)$ in case (d), acting by conjugation. 
In Section~\ref{sec:general} we recall the notion of {\it covariants} and their relation to the algebra $\mc[\mat]^U$ of $U$-invariants on $\mat$ (where $U$ is a maximal unipotent subgroup in $G$), and formulate Lemma~\ref{lemma:logic} underlying our strategy to transfer information on relations between basic covariants to 
give the ideal of $G$-stable subsets in $\mat$. Generators of the algebra of covariants on $\mat=\mc^{3\times 3}$ were determined by Tange \cite{tange}. 
In Section~\ref{sec:tange} we recall this result (and provide a natural interpretation of the generators). It turns out that the algebra of $U$-invariants on 
$\mat_1$ is isomorphic to a monomial subring of the three-variable polynomial ring, hence it is easy to determine its presentation, see Theorem~\ref{thm:covsonF}. 
From this we deduce Corollary~\ref{cor:main2}, describing a minimal generating system of $\ideal(\mat_1)$ as well as the $G$-module structure of the minimal degree non-zero homogeneous component of $\ideal(\mat_1)$. Moreover, in Section~\ref{sec:hilbseries} we derive the formal character of the  $G$-module $\mc[\mat_1]$, in particular, we compute the Hilbert series of the coordinate ring of $\mat_1$ as a rational function, see Corollary~\ref{cor:hilbertseries}. 
In Section~\ref{sec:symmetric} we show how the same method yields similar results (Corollary~\ref{cor:main3} and Corollary~\ref{cor:hilbseriessym})  for case (d): 
here the algebra of $U$-invariants on $\mat$ can be obtained from classial results on covariants of binary quartic forms.  
Finally, in Sections~\ref{sec:real} and \ref{sec:subdischer} we generalize some of the constructions of Section~\ref{sec:tange}  to arbitrary $n$. We derive some partial information on $\ideal((\mc^{n\times n})_k)$ for arbitrary $n$ and $k$,  and extend the results in 
\cite{domokos2}  on sum of squares presentations of subdiscriminants of real symmetric matrices to the case of $n\times n$ Hermitian matrices.

%%%%%%%%%%%%%%%%%%%%%%%%%%%%%%%%%%%%%

\section{Complex Zariski closure of the set of degenerate Hermitian matrices}\label{sec:zariski}  

The special linear group $SL(n,\mc)$ acts on $\mc^{n\times n}$ by conjugation. Two matrices in $\mc^{n\times n}$ are {\it similar} if they belong to the same $SL(n,\mc)$-orbit. 
A matrix in $\mc^{n\times n}$ is {\it diagonalizable} if it is similar to a diagonal matrix. It is well known that the subset of diagonalizable matrices is Zariski dense in $\mc^{n\times n}$. We need the following refinement: 

\begin{proposition}\label{prop:zariskidense1} 
The diagonalizable elements constitute a Zariski dense subset in $(\mc^{n\times n})_k$ for $k=0,1,\dots,n-1$. 
\end{proposition} 

\begin{proof} Note that any matrix in $(\mc^{n\times n})_k$ has at most $n-k$ distinct eigenvalues. If $A\in(\mc^{n\times n})_k$ has exactly $n-k$ distinct eigenvalues, then 
$m_A$ has no multiple roots, hence $A$ is diagonalizable. Now suppose that $A\in(\mc^{n\times n})_k$ is not diagonalizable, hence in particular it has  strictly less than $n-k$ distinct eigenvalues. Moreover, for some eigenvalue $\lambda$ of $A$ the root factor $x-\lambda$ has multiplicity $r \ge 2$ in the minimal polynomial $m_A$. 
We shall construct a polynomial map $\mc \to (\mc^{n\times n})_k$, $\varepsilon\mapsto A_{\varepsilon}$ such that $A_0=A$ and for all but finitely many $\varepsilon$ the matrix 
$A_{\varepsilon}$ has more eigenvalues than $A$. Since $m_A$ and the set of eigenvalues of $A$ is an invariant of the $SL(n,\mc)$-orbit of $A$, we may assume that $A$ is in Jordan normal form. By assumption on the minimal polynomial of $A$, it has a Jordan block $J_r(\lambda)$. In each such Jordan block of $A$ replace the $(1,1)$-entry by $\lambda+\varepsilon$; the resulting matrix is $A_{\varepsilon}$. 
When $\lambda+\varepsilon$ is not an eigenvalues of $A$, we have $m_{A_{\varepsilon}}=\frac{x-\lambda-\varepsilon}{x-\lambda}m_A$, 
so $A_{\varepsilon}\in(\mc^{n\times n})_k$, and  $A_{\varepsilon}$ has one more eigenvalues than $A$. 
Consequently $A$ is contained in the Zariski closure of the subset of those elements in 
$(\mc^{n\times n})_k$ that have more eigenvalues than $A$. By a descending induction on the number of distinct eigenvalues of $A$ one deduces the statement. 
 \end{proof}  

The complex orthogonal group 
\[O(n,\mc)=\{A\in\mc^{n\times n}\mid AA^T=I_n\}\]
acts by conjugation on $\mc^{n\times n}$, and $\symmat(n,\mc)$ is an invariant subspace. 
Two matrices are {\it orthogonally similar} if they belong to the same $O(n,\mc)$-orbit.  A matrix $B$ is {\it orthogonally diagonalizable} if 
it is orthogonally similar to a diagonal matrix (this forces $B\in\symmat(n,\mc)$). It is easy to see that the $O(n,\mc)$-orbit of a diagonal matrix coincides with its  
$SO(n,\mc)$-orbit, where 
\[SO(n,\mc)=\{A\in O(n,\mc)\mid \det(A)=1\}\]
is the {\it special orthogonal group}.  
 
\begin{proposition}\label{prop:zariskidense2} 
The orthogonally diagonalizable elements constitute a Zariski dense subset in $\symmat(n,\mc)_k$  for $k=0,1,\dots,n-1$. 
\end{proposition} 

\begin{proof} If $A\in\symmat(n,\mc)_k$ has $n-k$ distinct eigenvalues, then it is diagonalizable, hence by Theorem 4.4.13 in \cite{horn1} it is orthogonally diagonalizable. 
Thus by  induction on the number of distinct eigenvalues, it suffices to prove that if $A\in\symmat(n,\mc)_k$ has less than $n-k$ eigenvalues and is not diagonalizable, then it is contained in the Zariski closure of the subset of $\symmat(n,\mc)_k$ consisting of matrices having more eigenvalues than $A$. Since the action of $O(n,\mc)$ on $\symmat(n,\mc)$ preserves both the minimal polynomial and the number of eigenvalues of a matrix, in order to prove this claim it is sufficient to deal with $A$ taken from a particular set of $O(n,\mc)$-orbit representatives in $\symmat(n,\mc)$. 
Any matrix in $\mc^{n\times n}$ is similar to a symmetric matrix  (see Theorem 4.4.9 in \cite{horn1}), and if two symmetric matrices are similar, then they are orthogonally similar (see Corollary 6.4.19 in \cite{horn2}).  We recall from \cite{horn1} an explicit symmetric matrix in the similarity class of a Jordan block $J_r(\lambda)$. 
Denoting by $E_{i,j}$ the matrix unit with $(i.j)$-entry $1$ and zeroes everywhere else, we have $J_r(\lambda)=\lambda I_r+N_r$ where 
$N_r:=\sum_{j=1}^{r-1}E_{j,j+1}$ is the nilpotent Jordan block. Define $B_r:=\frac{1}{\sqrt 2}(I_r+i\sum_{s+t=r+1}E_{s,t})$ where $i$ is the imaginary complex unit with $i^2=-1$. 
We have $B_r\bar B_r=I_r$ and 
$B_rE_{s,t}\bar B_r=\frac 12(E_{s,t}+E_{r+1-s,r+1-t}+iE_{r+1-s,t}-iE_{s,r+1-t})$. 
This shows that $S_r(\lambda):=B_rJ_r(\lambda)\bar B_r$ is symmetric. 
For $\varepsilon\in\mc$ set 
\[S_{r,\varepsilon}(\lambda):=\begin{cases}  B_r(J_r(\lambda)-\varepsilon^2E_{m+1,m})\bar B_r,\mbox{ when }r=2m\\
B_r(J_r(\lambda)+\frac{\varepsilon^2}{2}(E_{m+1,m}+E_{m+2,m+1}))\bar B_r\mbox{ when }1<r=2m+1\\
S_1(\lambda+\varepsilon)\mbox{ when }r=1 \end{cases}\]
Then $S_{r,\varepsilon}(\lambda)$ is symmetric, and for $r>1$ its charateristic polynomial is 
$k_{S_{r,\varepsilon}(\lambda)}=(x-\lambda-\varepsilon)(x-\lambda+\varepsilon)(x-\lambda)^{r-2}
=\frac{(x-\lambda-\varepsilon)(x-\lambda+\varepsilon)}{(x-\lambda)^2}k_{S_r(\lambda)}$.  
 Assume now that $A\in \symmat(n,\mc)_k$ is not diagonalizable, and is block diagonal, with diagonal blocks of the form $S_l(\mu)$ with various $\mu\in\mc$ and $l\in\mn$. 
By assumption the minimal polynomial $m_A$ has a root factor $x-\lambda$ with multiplicity at least  $2$. 
Take for $A_{\varepsilon}$ the matrix obtained by replacing each block $S_r(\lambda)$ in $A$ by $S_{r,\varepsilon}(\lambda)$. 
Then $m_{A_{\varepsilon}}$ divides $\frac{(x-\lambda-\varepsilon)(x-\lambda+\varepsilon)}{(x-\lambda)^2}m_A$, so 
$A_{\varepsilon}\in\symmat(n,\mc)_k$.  Moreover, when none of $\lambda+\varepsilon$ and $\lambda-\varepsilon$ is an  eigenvalue of $A$, then $A_{\varepsilon}$ has one or two more eigenvalues than $A$. 
This shows that $A$ is contained in the Zariski closure of the subset of $\symmat(n,\mc)_k$ consisting of matrices with more eigenvalues than $A$. 
So our claim is proved. 
\end{proof} 

\begin{remark}\label{remark:euclidean} 
The proofs of Propositions~\ref{prop:zariskidense1} and \ref{prop:zariskidense2} show that for any $A\in(\mc^{n\times n})_k$ (respectively $A\in\symmat(n,\mc)_k$) there are diagonalizable (respectively orthogonally diagonalizable) elements in $(\mc^{n\times n})_k$ (respectively $\symmat(n,\mc)_k$) arbitrarily close to $A$ with respect to the euclidean metric. 
\end{remark} 

\begin{proposition}\label{prop:closure} \begin{enumerate}
\item[(i)] The Zariski closure of $\hermat(n)_k$ in the complex affine space $\mc^{n\times n}$ is $(\mc^{n\times n})_k$. 
\item[(ii)] The Zariski closure of $\symmat(n,\mr)_k$ in the complex affine space $\symmat(n,\mc)$ is $\symmat(n,\mc)_k$. 
\end{enumerate}
\end{proposition} 

\begin{proof} (i) The {\it special unitary group} 
\[SU(n):=\{A\in\mc^{n\times n}\mid A\bar A^T=I_n,\quad \det(A)=1\}\] 
is Zariski dense in the complex linear algebraic group $SL(n,\mc)$. 
Note that the subset $\hermat(n)_k$ in $\mc^{n\times n}$ is $SU(n)$-stable, hence its Zariski closure is $SL(n,\mc)$-stable. 
Therefore by Proposition~\ref{prop:zariskidense1} it is sufficient to show that the Zariski closure of $\hermat(n)_k$ contains the set  $X$ of all complex 
diagonal matrices with at most $n-k$ distinct diagonal entries.  Let $L$ be an  irreducible component of $X$. Then  $L$ is an $n-k$-dimensional linear  subspace, 
spanned by its intersection with the space $D$ of  real diagonal matrices. Now $L\cap D\subset \hermat(n)_k$, and the Zariski closure of the real linear subspace $L\cap D$ is obviously its  $\mc$-linear span $L$. Thus $X$ is contained in the Zariski closure of $\hermat(n)_k$. 

The proof of (ii) is similar: the real special orthogonal group $SO(n)$ is Zariski dense in $SO(n,\mc)$, hence the Zariski closure of $\symmat(n,\mr)_k$ is $SO(n,\mc)$-stable. 
Now use Proposition~\ref{prop:zariskidense2} and conclude in the same way as above.  
\end{proof}

%%%%%%%%%%%%%%%%%%%%%%%%%

\section{Covariants and $G$-stable ideals}\label{sec:general} 

Let $G$ be a connected reductive linear algebraic group over the base field $\mc$ (like $SL_n(\mc)$ or $SO(n,\mc)$). 
Fix a maximal unipotent subgroup $U$ in $G$, and a maximal torus $T$ in $G$ normalizing $U$.  
We need to recall some basic facts from highest weight theory (cf. e.g. \cite{fulton-harris}, \cite{goodman-wallach}, \cite{procesi}): by a {\it $G$-module} 
we mean a rational $G$-module. Any $G$-module $V$ is spanned by $T$-eigenvectors. A (non-zero) $T$-eigenvector $v$ is called a {\it weight vector}, and the 
character $\lambda:T\to\mc^\times$ given by $t\cdot v=\lambda(t)v$ is called its weight. A $U$-invariant weight vector is called a {\it highest weight vector}. 
A highest weight vector generates an irreducible $G$ submodule. Moreover, an irreducible $G$-module contains a unique (up to scalar multiples) highest weight vector. 

Our proof of Corollary~\ref{cor:main2} and \ref{cor:main3} is based on the following general observation. 
Let $M$ be an affine $G$-variety with coordinate ring $\mc[M]$. It is a $G$-module via $(g\cdot f)(x)=f(g^{-1}x)$ for $g\in G$, $f\in\mc[M]$, $x\in M$. 
The algebra $\mc[M]^U$ of $U$-invariant polynomial functions on $M$ is finitely generated by \cite{hadziev} (see Theorem 9.4 in \cite{grosshans} or \cite{donkin}). Let $u_1,\dots,u_r$ be generators of the algebra $\mc[M]^U$. 

\begin{lemma}\label{lemma:logic} For any Zariski closed $G$-stable subset $X$ in $M$, the vanishing ideal $\ideal(X)$ is generated as a $G$-stable ideal in 
$\mc[M]$ by 
$f_j(u_1,\dots,u_r)$, $j=1,\dots,m$, where $f_1,\dots,f_m$ generate as an ideal in the $r$-variable  polynomial ring  the kernel  of the $\mc$-algebra  homomorphism 
$\varphi:\mc[x_1,\dots.x_r]\to\mc[X]^U$ given by  $x_i\mapsto u_i\vert_{X}$ (the restriction  of $u_i$ to $X$), $i=1,\dots,r$. 
\end{lemma}

\begin{proof} Denote by $\eta$ the $\mc$-algebra homomorphism $\mc[M]^U\to \mc[X]^U$ given by restriction of functions on $M$ to $X$. 
Obviously we have $\varphi=\eta\circ\Psi$,  where $\Psi:\mc[x_1,\dots,x_r]\to\mc[M]^U$ is the $\mc$-algebra surjection given by $x_i\mapsto u_i$ ($i=1,\dots,r$). 
Hence $\ker(\eta)=\Psi(\ker(\varphi))$. On the other hand, clearly $\ker(\eta)=\ideal(X)^U$. 
Recall that any $G$-submodule of  $\mc[M]$ is the sum of  finite dimensional irreducible $G$-submodules, each summand containing a non-zero $U$-invariant element (a highest weight vector).  
Therefore any $G$-submodule of $\mc[M]$ is generated by its $U$-invariant elements. In particular, $\ideal(X)$ is generated by $\ideal(X)^U$ as a $G$-module. 
\end{proof} 

\begin{remark}\label{remark:surjective} The map $\eta:\mc[M]^U\to\mc[X]^U$ is  surjective. Indeed, the maximal torus 
$T$ acts rationally on $\mc[M]^U$ and on $\mc[X]^U$, and these spaces are spanned by  {\it weight vectors} (i.e. $T$-eigenvectors).  Thus  it is sufficient to show that any weight vector $h\in\mc[X]^U$ is contained in the image of $\eta$. Since $h$ is $U$-invariant, it is a highest weight vector in $\mc[X]$, hence  generates an irreducible $G$-submodule $V$ in $\mc[X]$. Thus $\ideal(X)$ has an irreducible $G$-module direct complement $V'$ in the inverse image of $V$ under the natural surjection $\mc[M]\to\mc[X]$. Take a highest weight vector $h'$ in $V'$, so $h'\in \mc[M]^U$, and $\eta(h')$ is a nonzero scalar multiple of $h$.  
\end{remark} 

By a {\it  covariant $f$ on $M$} we mean a non-zero $G$-equivariant polynomial map $f:M\to V$, where $V$ is a finite dimensional (rational) $G$-module. The non-zero covariant $f$ is {\it irreducible} if $V$ is an irreducible $G$-module. In this case the 
comorphism  of $f$ restricts to an embedding $f^\star$ of the $G$-module $V^\star$ into the coordinate ring $\mc[M]$, and we shall denote by $f^U\in\mc[M]^U$ the unique (up to scalar multiples) highest weight vector in $f^\star(V^\star)$. 
Conversely, a non-zero $T$-eigenvector in $\mc[M]^U$ generates an irreducible $G$-submodule $W$ in $\mc[M]$, and  the map $M\to W^\star$ sending 
$m\in M$ to the linear functional $W\to \mc$, $w\mapsto w(m)$ is an irreducible covariant.   So an irreducible covariant determines (up to scalar multiples) a non-zero $T$-eigenvector in $\mc[M]^U$, and vice versa. 
The algebra $\mc[M]^U$ is sometimes called therefore  the {\it  algebra of covariants} on $M$. 
We shall write $\covariants_{G}(M,V)$ for the set of covariants $f:M\to V$; it is naturally a module over the algebra $\mc[M]^{G}$ of polynomial invariants on $M$. 

%%%%%%%%%%%%%%%%%%%%%

\section{Covariants of $3\times 3$ matrices} \label{sec:tange} 

In Sections~\ref{sec:tange} and \ref{sec:hilbseries}  set $\mat:=\mc^{3\times 3}$ and  $G:=SL(3,\mc)$ acting by conjugation on $\mat$. 
We take the maximal unipotent subgroup $U$ of $G$ consisting of the unipotent upper triangular matrices, normalized by  the maximal torus $T$  consisting of the diagonal matrices in $G$. 
Generators of the algebra $\mc[\mat]^U$ were determined by Tange \cite{tange}, Section 3. Here we give a natural interpretation of all the generators, 
by presenting some natural covariants $f$ on $\mat$ such that the corresponding  $f^U$ (with the notation introduced in Section~\ref{sec:general}) 
provide the generators found in \cite{tange}.  
We shall identify the group $\mathrm{Char}(T)$ of rational characters of the maximal torus $T$ in  $G$ with $\mz^2$, such that 
$\lambda\in\mz^2$ corresponds to  the  character of $T$  given by $\mathrm{diag}(z_1,z_2,z_1^{-1}z_2^{-1})\mapsto z_1^{\lambda_1}z_2^{\lambda_2}$. The possible highest weights correspond to $\{\lambda\in\mz^2\mid \lambda_1\ge\lambda_2\ge 0\}$, and 
denote by $V^{\lambda}$ the irreducible $G$-module with highest weight $\lambda$. The $G$-module $V^{(2,1)}$ can be realized as 
\[V^{(2,1)}\cong \tzmat:=\{A\in\mat\mid \tr(A)=0\}\]
We start with the  covariant $\mat\to\tzmat$ given by  
\begin{equation}\label{eq:tange1}\tange_1:A\mapsto A-\frac 13\tr(A)I_3 \end{equation}
where $\tr$ is the usual trace function. 
Define a second covariant $\mat\to\tzmat$ by 
\begin{equation}\label{eq:tange2} 
\tange_2:=\tange_1\circ \tilde\tange_2\circ\tange_1\quad \mbox{ where }\tilde\tange_2:\tzmat\to\mat,\quad A\mapsto A^2 \end{equation}
Recall that the defining representation of $G$ on $\mc^3$ is irreducible and is isomorphic to $V^{(1,0)}$, its dual is $(\mc^3)^\star\cong V^{(1,1)}$. 
The symmetric powers of $\mc^3$ and $(\mc^\star)^3$ are also irreducible, we have 
$\mathrm{S}^3(\mc^3)\cong V^{(3,0)}$ and $\mathrm{S}^3(\mc^3)^\star\cong V^{(3,3)}$. 
Think of $\mathrm{S}^3(\mc^3)^\star$ as the space of homogeneous cubic polynomial functions on $\mc^3$, and define a covariant 
\begin{equation}\label{eq:tange6}\tange_3:\mat\to\mathrm{S}^3(\mc^3)^\star, \quad  
A\mapsto (\underline{x} \mapsto \det (\underline{x}\vert A\underline{x}\vert A^2\underline{x}))
\end{equation}
where for $\underline{x}\in \mc^3$ and $A\in\mat$ we write $(\underline{x}\vert A\underline{x}\vert A^2\underline{x}))$ for the  $3\times 3$ matrix 
whose columns are $\underline{x}$, $A\underline{x}$, $A^2\underline{x}$ and $\det$ is the determinant. 
For $g\in G$ we have 
\begin{eqnarray*}(\tange_3(gAg^{-1}))(\underline{x})&=&\det(\underline{x}\vert gAg^{-1}\underline{x}\vert gA^2g^{-1}\underline{x})
=\det(g)\det(g^{-1}\underline{x}\vert Ag^{-1}\underline{x}\vert A^2g^{-1}\underline{x}) 
\\ &=&\tange_3(A)(g^{-1}\underline{x})=(g\cdot\tange_3(A))(\underline{x})
\end{eqnarray*}
showing that $\tange_3$ is indeed a covariant. Moreover, it is  non-zero (e.g. take for $A$ a diagonal matrix with distinct eigenvalues), hence is an irreducible covariant. 
Identify $(\mc^3)^\star$ with the space of row vectors $\{\underline{x}^T\mid \underline{x}\in\mc^3\}$ in the standard way.  
Think of $\mathrm{S}^3(\mc^3)$ as the space of homogeneous cubic polynomial functions on $(\mc^3)^\star$, and similarly to the construction of $\tange_3$,  
define the irreducible covariant 
\begin{equation}\label{eq:tange7}\tange_4:\mat\to\mathrm{S}^3(\mc^3),\quad A\mapsto (\underline{x}^T\mapsto \det\left(\begin{array}{c}\underline{x}^T \\\underline{x}^TA  \\\underline{x}^TA^2\end{array}\right))
\end{equation}
It is well known that the algebra $\mc[\mat]^{G}$ of polynomial invariants is generated by the following three algebraically independent elements: 
\[\inv_1:A\mapsto \tr(A), \quad \inv_2:A\mapsto \frac 16\tr(\tange_1(A)^2),\quad \inv_3:A\mapsto \frac 12\det(\tange_1(A))\]
 (the scalars $\frac 16$ and $\frac 12$ above are chosen in order to make certain later formulae simpler).

\begin{proposition}\label{prop:tange}
The algebra $\mc[\mat]^U$ is generated by the seven elements $\inv_i$ ($i=1,2,3$) and  $\tange_j^U$ ($j=1,2,3,4$).  
\end{proposition} 

\begin{proof} 
This is a restatement of  Proposition 2 from \cite{tange} giving seven explicit $T$-eigenvectors generating $\mc[\mat]^U$. To see this  one just has to write down explicit expressions 
for the $\tange_j^U$ in terms of the coordinate functions on $\mat$. 
\end{proof} 

We shall view $\mc[\tzmat]$ as a subalgebra of $\mc[\mat]$ via the embedding $f\mapsto f\circ \tange_1$. 
Clearly $\mc[\mat]$ is a polynomial ring over $\mc[\tzmat]$ generated by $\inv_1$. 
Moreover, 
\begin{equation}\label{eq:fcircc1}\mbox{for each } f\in \{\inv_2,\inv_3,\tange_1,\tange_2,\tange_3,\tange_4\}\mbox{ we have that }f=f\circ\tange_1
\end{equation} hence $f^U\in\mc[\tzmat]$. 
Thus Proposition~\ref{prop:tange} can be restated as follows: 

\begin{proposition}\label{prop:tange2}
We have $\mc[\mat]^U=\mc[\tzmat]^U[\inv_1]$ and 
$\mc[\tzmat]^U$ is generated by $\inv_2,\inv_3,\tange_1^U,\tange_2^U,\tange_3^U,\tange_4^U$. 
\end{proposition}

\begin{proposition} \label{prop:coordfunctions} 
The coordinate functions  of 
$\tange_3$ and $\tange_4$ 
are contained in the $\mc$-subspace of $\mc[\mat]$ spanned by all the coordinate functions of $\mathcal{P}_1:\mat\to\bigwedge^3\mat$, 
$A\mapsto I_3\wedge A\wedge A^2$. 
\end{proposition} 

\begin{proof} 
This follows from the Cauchy-Binet formula and the following two matrix equalities, where  $e_1,e_2,e_3$ are the standard basis vectors in $\mc^3$, $a_1,a_2,a_3$ are the columns of a $3\times 3$ matrix $A$, $b_1,b_2,b_3$ are the columns of a $3\times 3$ matrix $B$, and $\underline{x}\in\mc^3$: 
\[\left(\begin{array}{ccc}x_1I_3 & x_2I_3 & x_3I_3 \end{array}\right)_{3\times 9}\left(\begin{array}{ccc}e_1 & a_1 & b_1 \\e_2 & a_2 & b_2 \\e_3 & a_3 & b_3\end{array}\right)_{9\times 3}=\left(\begin{array}{ccc}\underline{x} & A\underline{x} & B\underline{x}\end{array}\right)_{3\times 3}\]  
\[\left(\begin{array}{ccc}\underline{x}^T & 0 & 0 \\ 0 & \underline{x}^T & 0 \\0 & 0 & \underline{x}^T\end{array}\right)_{3\times 9}
\left(\begin{array}{ccc}e_1 & a_1 & b_1 \\e_2 & a_2 & b_2 \\e_3 & a_3 & b_3\end{array}\right)_{9\times 3}=
\left(\begin{array}{ccc}\underline{x} & A^T\underline{x} & B^T\underline{x}\end{array}\right)_{3\times 3}
=\left(\begin{array}{c}\underline{x}^T \\\underline{x}^TA \\\underline{x}^TB\end{array}\right)_{3\times 3}^T\] 
\end{proof} 

Now we turn to the affine subvariety $\mat_1\subset\mat$. 
Restriction of functions on $\mat$ to $\mat_1$ gives the natural surjection 
\[\mc[\mat]\to\mc[\mat_1]\] 
onto the coordinate ring $\mc[\mat_1]=\mc[\mat]/\ideal(\mat_1)$ of the affine algebraic variety $\mat_1$. 
We shall write $\overline{\inv_j}$, $\overline{\tange_j}$, respectively $\overline{\tange_j^U}$ for the restriction to $\mat_1$ of 
$\inv_i$, $\tange_j$, respectively $\tange_j^U$. 
The covariant $\overline{\tange_1}$ maps $\mat_1$ onto 
\[\tzmat_1:=\mat_1\cap\tzmat \]  
hence induces an embedding of 
$\mc[\tzmat_1]$ as a subalgebra of $\mc[\mat_1]$. 
Furthermore, $\mat_1=\tzmat_1\oplus\mc I_3$, hence 
\begin{equation}\label{eq:C[N1][d1]} 
\mc[\mat_1]=\mc[\tzmat_1][\overline{\inv_1}]
\end{equation}
is a polynomial ring generated by $\overline{\inv_1}$ over the subalgebra $\mc[\tzmat_1]$, and 
$\overline{\inv_2},\overline{\inv_3},\overline{\tange_i^U}
\in\mc[\tzmat_1]$.

\begin{proposition}\label{prop:firstrelations} 
We have the following equalities for covariants on $\mat_1$: 
\[\overline{\tange_3}=0,\quad \overline{\tange_4}=0, \quad 
\overline{\inv_2}^3=\overline{\inv_3}^2,\quad \overline{\inv_3}\overline{\tange_1}=\overline{\inv_2}\overline{\tange_2}\]
(the last equality is understood in the $\mc[\mat_1]^G$-module $\covariants_G(\mat_1,\tzmat)$). 
\end{proposition} 

\begin{proof} The equalities $\overline{\tange_3}=0=\overline{\tange_4}$ follow from Proposition~\ref{prop:coordfunctions} and the fact that $\mathcal{P}_1$ maps $\mat_1$ to zero. Since diagonalizable elements in $\mat_1$ constitute a Zariski dense subset in $\mat_1$ by Proposition~\ref{prop:zariskidense1}, 
it is sufficient to check vanishing of the polynomial maps 
$\inv_2^3-\inv_3^2$ and $\inv_3\tange_1-\inv_2\tange_2$ 
on diagonalizable elements in $\mat_1$. 
 Therefore by \eqref{eq:fcircc1} and the covariance property  it is sufficient to check vanishing of the above covariants on the diagonal matrices 
 $D(z):=\mathrm{diag}(z,z,-2z)$ where $z\in\mc$. 
Now we have 
\[\inv_2(D(z))=z^2,\quad \inv_3(D(z))=-z^3,\quad \tange_1(D(z))=D(z),\quad \tange_2(D(z))=-zD(z)\] 
so the desired relations obviously hold. 
\end{proof} 

Recall that the elements $\tange_i^U$ are determined only up to non-zero scalar multiples; according to 
Proposition~\ref{prop:firstrelations} it is possible to normalize $\tange_1^U$ and $\tange_2^U$ so that 
the equality 
\begin{equation}\label{eq:normalization}\overline{\inv_3}\overline{\tange_1^U}=\overline{\inv_2}\overline{\tange_2^U}\end{equation}
holds, and from now on we assume that $\tange_1^U$ and $\tange_2^U$ were chosen so that \eqref{eq:normalization} holds. 
The standard $\mn_0$-grading on the polynomial algebra  $\mc[\mat]$ and the grading by the group $\mathrm{Char}(T)=\mz^2$ of rational characters defined by the action of the maximal torus $T\subset G$ can be combined to a bigrading by $\mn_0\times\mathrm{Char}(T)$: we say that $f\in \mc[\mat]$ is {\it bihomogeneous of bidegree} 
$\bideg(f)=(n,\lambda)$ if $f(zA)=z^nf(A)$ for all $A\in\mat$ and $z\in\mc$, and $t\cdot f=t_1^{\lambda_1}t_2^{\lambda_2}f$ for any 
$\mathrm{diag}(t_1,t_2,t_1^{-1}t_2^{-1})\in T$.  
Clearly the algebras 
$\mc[\mat]^U$, $\mc[\mat_1]$, $\mc[\mat_1]^U$, $\mc[\tzmat_1]^U$ all inherit the bigrading from $\mc[\mat]$. 

In the following statement  $\mc[z^2,z^3,D,zD]$ stands for the subalgebra of the two-variable polynomial ring $\mc[z,D]$ generated by the monomials  
$z^2,z^3,D,zD$ (the notation $z,D$ for the indeterminates is inspired by the proof of Proposition~\ref{prop:firstrelations}), and $\mc[x_0,x_1,x_2,x_3,x_4]$ is a five-variable polynomial algebra.

\begin{theorem}\label{thm:covsonF}  
(i) The algebra $\mc[\mat_1]^U$ is a polynomial ring generated by $\overline{\inv_1}$ over $\mc[\tzmat_1]^U$. 
 
(ii) There is a $\mc$-algebra isomorphism $\eta:\mc[z^2,z^3,D,zD]\to \mc[\tzmat_1]^U$ with 
\[\eta:z^2\mapsto\overline{\inv_2}, \quad z^3\mapsto \overline{\inv_3},\quad  
D\mapsto\overline{\tange_1^U}, \quad  zD\mapsto \overline{\tange_2^U}. \]

(iii) The kernel of the natural surjection 
$\varphi:\mc[x_0,x_1,x_2,x_3,x_4]\to \mc[\mat_1]^U$, 
$x_0\mapsto \overline{\inv_1}$, $x_1\mapsto \overline{\inv_2}$, $x_2\mapsto \overline{\inv_3}$, $x_3\mapsto \overline{\tange_1^U}$, $x_4\mapsto \overline{\tange_2^U}$ 
is generated as an ideal by 
\[x_1^3-x_2^2,\qquad x_1x_4-x_2x_3,\qquad x_4^2-x_1x_3^2,\qquad x_2x_4-x_1^2x_3.\]
\end{theorem} 

\begin{proof}  Statement (i) follows from \eqref{eq:C[N1][d1]}.  
As explained in Remark~\ref{remark:surjective}, the natural surjection $\mc[\tzmat]\to\mc[\tzmat_1]$ restricts to a surjection 
$\mc[\tzmat]^U\to\mc[\tzmat_1]^U$.  Hence by Propositions~\ref{prop:tange2} and \ref{prop:firstrelations} 
$\mc[\tzmat_1]^U$ is generated by $\overline{\inv_2},\overline{\inv_3},\overline{\tange_1^U},\overline{\tange_2^U}$. The variety $\tzmat_1$ is irreducible, as by Proposition~\ref{prop:zariskidense1}  it is the Zariski closure of  $G \cdot \{D(z)\mid z\in \mc\}$ (with the notation of the proof of Proposition~\ref{prop:firstrelations}).  
Thus the  coordinate ring $\mc[\tzmat_1]$ is a domain, and by  \eqref{eq:normalization}  we have the equality  
$\overline{\tange_2^U}=\overline{\tange_1^U}\overline{\inv_3}/\overline{\inv_2}$
in the function field $\mc(\tzmat_1)$. 
The proof of Proposition~\ref{prop:firstrelations} shows that the map $z^2\mapsto \overline{\inv_2}$, $z^3\mapsto \overline{\inv_3}$ extends to a $\mc$-algebra isomorphism 
$\mc[z^2,z^3]\to \mc[\overline{\inv_2},\overline{\inv_3}]\subset \mc[\tzmat_1]$. This extends to a $\mc$-algebra surjection 
$\tilde\eta:\mc[z^2,z^3,D]\to \mc[\overline{\inv_2},\overline{\inv_3},\overline{\tange_1^U}]$ with $D\mapsto \overline{\tange_1^U}$. 
We claim that $\tilde\eta$ is an isomorphism. Indeed, define a bigrading on the polynomial algebra $\mc[z,D]$ by setting 
$\bideg(z):=(1,(0,0))$ and $\bideg(D):=(1,(2,1))$. Then $\tilde\eta$ is a homomorphism of bigraded algebras, so $\ker(\tilde\eta)$ is spanned by bihomogeneous elements. 
Now observe that the bihomogeneous components of $\mc[z^2,z^3,D]$ are one-dimensional, each is spanned by a monomial  
$(z^2)^i(z^3)^jD^k$, and these monomials are not mapped to zero, since $\overline{\inv_2}$, $\overline{\inv_3}$, $\overline{\tange_1^U}$ are non-zero, and  
$ \mc[\overline{\inv_2},\overline{\inv_3},\overline{\tange_1^U}]$ is a domain. 
The isomorphism $\tilde\eta$ extends to an isomorphism between the fields of fractions of $\mc[z^2,z^3,D]$ and 
$\mc[\overline{\inv_2},\overline{\inv_3},\overline{\tange_1^U}]$, and this latter field isomorphism restricts to the desired  $\mc$-algebra isomorphism 
$\eta:\mc[z^2,z^3,D,zD]\to \mc[\overline{\inv_2},\overline{\inv_3},\overline{\tange_1^U},\overline{\tange_1^U}\overline{\inv_3}/\overline{\inv_2}]$. 
Thus (ii) is proved. 

To prove (iii),  by  (ii) it is sufficient to show that the given four polynomials generate the kernel of the natural surjection 
$\phi: \mc[x_1,x_2,x_3,x_4]\to\mc[z^2,z^3,D,zD]$ given by $x_1\mapsto z^2$, $x_2\mapsto z^3$, $x_3\mapsto D$, $x_4\mapsto zD$. 
The given four polynomials are indeed in the kernel of $\phi$, and it is easy to see that modulo the ideal generated by them, any monomial in 
$\mc[x_1,x_2,x_3,x_4]$ can be rewritten as a linear combination of the monomials 
\[\{x_3^ix_4,\quad x_1^ix_3^j, \quad x_1^ix_2x_3^j\quad \mid \quad i,j=0,1,\dots\}.\] 
Now $\phi$ maps bijectively the above set of monomials onto 
$\{z^kD^l\mid (k,l)\neq (1,0)\}$, which is a basis of $\mc[z^2,z^3,D,zD]$.  
This implies the claim. 
\end{proof}

\begin{corollary}\label{cor:sl3stable} 
As a $G$-stable ideal, $\ideal(\mat_1)$ is generated by $\tange_3^U$ and $\tange_4^U$. 
\end{corollary} 

\begin{proof} We apply Lemma~\ref{lemma:logic}: by Propositions~\ref{prop:tange2}, \ref{prop:firstrelations} and by Theorem~\ref{thm:covsonF} we conclude that $\ideal(\mat_1)$ is generated as a $G$-stable ideal by $\tange_3^U$, $\tange_4^U$, $\inv_2^3-\inv_3^2$, $\inv_2\tange_2^U-\inv_3\tange_1^U$, $(\tange_2^U)^2-\inv_2(\tange_1^U)^2$, 
$\inv_3\tange_2^U-\inv_2^2\tange_1^U$. 
It is easy to verify by computer (we used the computer algebra system \cite{cocoa}) that the latter four elements of $\mc[\mat]$ are contained in the ideal generated by the coordinate functions of $\tange_3$ (the linear span of these coordinate functions is the $G$-module generated by $\tange_3^U$), hence the result follows. 
\end{proof} 

\begin{corollary}\label{cor:main2}
(i) The ideal $\ideal(\mat_1)$ is generated by its degree $3$ homogeneous component $\ideal(\mat_1)_3$. 

(ii) The  $20$ coordinate functions of $\tange_3$ and $\tange_4$ constitute a $\mc$-basis in $\ideal(\mat_1)_3$.  

(iii) As a  $G$-module $\ideal(\mat_1)_3$ is isomorphic to $\mathrm{S}^3(\mc^3)\oplus \mathrm{S}^3(\mc^3)^\star$. 

(iv)  The coordinate functions of $\mathcal{P}_1:\mat\to\bigwedge^3\mat$ span $\ideal(\mat_1)_3$. 
\end{corollary}

\begin{proof} Since $\tange_3^U$ and $\tange_4^U$ are homogeneous of degree three, it follows trivially from Corollary~\ref{cor:sl3stable} that $\ideal(\mat_1)$ is generated by its degree three homogeneous component, so (i)  is proved. 
The covariants $\tange_3$ and $\tange_4$ are non-zero, irreducible, and map $\mat$ into non-isomorphic $G$-modules. It follows that their coordinate functions are linearly independent, so both (ii) and (iii) hold by Corollary~\ref{cor:sl3stable} and by construction of $\tange_3$, $\tange_4$. 
Finally, (iv) follows from (ii) and Proposition~\ref{prop:coordfunctions}.  
\end{proof} 

\begin{remark} In \cite{kostant-wallach} for any complex simple Lie group $G$  the authors construct a $G$-submodule in the minimal degree non-zero homogeneous component of the vanishing ideal of the subset of singular elements in the Lie algebra of $G$, and determine its $G$-module structure. For the special case $G=SL(n,\mc)$ this subspace coincides with the space spanned by 
the coordinate functions of $\mathcal{P}_1$. 
\end{remark} 

%%%%%%%%%%%%%%%%%%%%%%%%%%%%%%%%%%%

\section{Hilbert series}\label{sec:hilbseries}

Following \cite{drensky-etal} we introduce the {\it graded multiplicity series} of $\mc[\mat_1]$ as follows: 
\begin{equation}\label{eq:multseries}
M(\mc[\mat_1];q_1,q_2,t):=\sum_{d=0}^\infty\sum_{\lambda\in\mathrm{Char}(T)}m(d,\lambda)q_1^{\lambda_1}q_2^{\lambda_2}t^d
\in\mz[q_1,q_2][[t]]
\end{equation}
where $m(d,\lambda)$ denotes the multiplicity of the irreducible $G$-module $V^{\lambda}$ as a summand in the degree $d$ homogeneous component 
of $\mc[\mat_1]$. 

\begin{corollary}\label{cor:multseries} 
We have the equality 
 \[M(\mc[\mat_1];q_1,q_2,t)=\frac{1-t+t^2+q_1^2q_2t^2-q_1^2q_2t^3}{(1-t)^2(1-q_1^2q_2t)}\]
\end{corollary} 

\begin{proof} By the discussion at the beginning of Sections~\ref{sec:general} and  \ref{sec:tange},   
$M(\mc[\mat_1];q_1,q_2,t)$ is nothing but the bigraded Hilbert series of $\mc[\mat_1]^U$ (with respect to the bigrading by $\mn_0\times\mathrm{Char}(T)$ introduced before 
Theorem~\ref{thm:covsonF}).  
By Theorem~\ref{thm:covsonF} (i) and (ii), $\overline{\inv_1}^i\eta(z^jD^k)$ where $i,j,k\in\mn_0$, $(j,k)\ne (1,0)$ is a $\mc$-vector space basis in 
$\mc[\mat_1]^U$, and the basis element corresponding to $(i,j,k)$ is bihomogeneous of  bidegree $(i+j+k,(2k,k))$. 
Consequently, $M(\mc[\mat_1];q_1,q_2,t)=\frac{1}{(1-t)^2(1-q_1^2q_2t)}-\frac{t}{1-t}$. 
\end{proof} 

The \emph{Hilbert series}  of a multigraded vector space in general is the generating function  of the dimensions of its multihomogeneous components. 
In particular,  the  Hilbert series of the bigraded algebra  $\mc[\mat_1]$ is 
\begin{equation}\label{eq:hilbertseries}
H(\mc[\mat_1];q_1,q_2,t):=\sum_{d=0}^\infty\sum_{\lambda\in\mathrm{Char}(T)}a(d,\lambda)q_1^{\lambda_1}q_2^{\lambda_2}t^d
\in\mz[q_1^{\pm 1},q_2^{\pm 1}][[t]]
\end{equation} 
where $a(d,\lambda)$ is the multiplicity of the $1$-dimensional $T$-module with character $\lambda$ in the degree $d$ homogeneous component of 
$\mc[\mat_1]$. 
The series \eqref{eq:multseries} and \eqref{eq:hilbertseries} are related by 
\[H(\mc[\mat_1];q_1,q_2,t)=\sum_{d=0}^\infty\sum_{\lambda\in\mathrm{Char}(T)}m(d,\lambda)\tr(\mathrm{diag}(q_1,q_2,q_1^{-1}q_2^{-1})\big\vert_{V^{\lambda}})t^d\]
where $\mathrm{diag}(q_1,q_2,q_1^{-1}q_2^{-1})\big\vert_{V^{\lambda}}$ is the linear transformation of $V^{\lambda}$ corresponding to 
$\mathrm{diag}(q_1,q_2,q_1^{-1}q_2^{-1})\in T\subset G$ under the representation  on $V^{\lambda}$. 
Setting $q_3:=q_1^{-1}q_2^{-1}$ and denoting by $S_3$ the symmetric group of degree $3$, we have 
\[\tr(\mathrm{diag}(q_1,q_2,q_1^{-1}q_2^{-1})\big\vert_{V^{\lambda}})=
\frac{\sum_{\pi\in S_3}\mathrm{sign}(\pi)q_{\pi(1)}^{\lambda_1+2}q_{\pi(2)}^{\lambda_2+1}}{(q_1-q_2)(q_1-q_3)(q_2-q_3)}\]
Consequently, still using the notation $q_3:=q_1^{-1}q_2^{-1}$ we have 
\[H(\mc[\mat_1];q_1,q_2,t)=\sum_{\pi\in S_3}\mathrm{sign}(\pi)\frac{q_{\pi(1)}^2q_{\pi(2)}M(\mc[\mat_1];q_{\pi(1)},q_{\pi(2)},t)}{\prod_{1\le i<j\le 3}(q_i-q_j)}\]
from which (after substituting $q_1=q_2=1$) one can easily compute the ordinary Hilbert series 
\[H(\mc[\mat_1];t):=\sum_{d=0}^\infty \dim_{\mc}(\mc[\mat_1]_d)t^d\] 
where $\mc[\mat_1]_d$ stands for the degree $d$ homogeneous component of the graded algebra $\mc[\mat_1]$:   

\begin{corollary}\label{cor:hilbertseries} 
We have the equality 
\[H(\mr[\hermat(3)_1];t)=H(\mc[\mat_1];t)=\frac{1+3t+6t^2-10t^3+10t^4-5t^5+t^6}{(1-t)^6}\] 
\end{corollary}

%%%%%%%%%%%%%%%%%%%%%%%%%%%%%%%

\section{Symmetric  $3\times 3$ matrices}\label{sec:symmetric} 

In this section set $\mat:=\symmat(3,\mc)$ and $G:=SO(3,\mc)$ acting by conjugation on $\mat$. Again denote $\tzmat$ the subset of trace zero 
matrices in $\mat$, and $\tzmat_1:=\mat_1\cap\tzmat$. We may restrict the covariants on $\mc^{3\times 3}$ introduced in Section~\ref{sec:tange} to its subspace $\mat$ of symmetric complex $3\times 3$ matrices; we keep the same notation $\inv_i$, $\tange_j$ for the resulting 
$G$-equivariant polynomial maps on $\mat$. An essential difference compared to the case of $\mc^{3\times 3}$ is that now $\tange_4=\tange_3$. 
Moreover, $\mathrm{S}^3(\mc^3)^\star$ is not an irreducible $\mat_1$-module. 
The maximal torus $T$ in $\mat_1$ has rank $1$, i.e. $T\cong \mc^\times$, so $\mathrm{Char}(T)=\mz$, where the character $T\to\mc^\times$, $t\mapsto t^n$ is identified with $n\in \mz$. The possible highest weights are the non-negative integers, we shall denote by $V^{(n)}$ the irreducible $G$-module with highest weight $n$; 
it has dimension $2n+1$, and for $t\in T=\mc^\times$ we have $\tr(t\big\vert_{V^{(n)}})=t^n+t^{n-1}+\dots+t^{-n}$.   
With this notation we have 
\[\mathrm{S}^3(\mc^3)^\star= V^{(3)}+V^{(1)}\] 
where $V^{(3)}$ is the kernel of the Laplace operator $\Delta:=\sum_{i=1}^3\frac{\partial ^2}{\partial  x_i^2}$ restricted to $\mathrm{S}^3(\mc^3)^\star$. 

\begin{proposition}\label{prop:laplace} 
The covariant $\tange_3$ is non-zero and maps $\mat$ into $\ker(\Delta\big\vert_{\mathrm{S}^3(\mc^3)^\star})$.  
\end{proposition} 

\begin{proof} 
Since $\Delta$ is a $G$-equivariant operator, and by Proposition~\ref{prop:zariskidense2} there is a Zariski dense subset in $\mat_1$ consisting of  
$G$-orbits of diagonal matrices, it suffices to show that $\Delta(\tange_3(A))=0$ for any diagonal $A\in \mat$. Now  
we have $\tange_3(\mathrm{diag}(a_1,a_2,a_3))=(a_2-a_1)(a_3-a_1)(a_3-a_2)x_1x_2x_3$. This shows that $\tange_3$ is non-zero, and since $\Delta(x_1x_2x_3)=0$, 
the second claim also follows.  
\end{proof} 

From now on we shall view $\tange_3$ as an irreducible covariant $\tange_3:\mat\to V^{(3)}$. 
Moreover,  
$\tange_1,\tange_2:\mat\to\tzmat\cong V^{(2)}$ 
are independent irreducible covariants. 
The covariant $\tange_1$ induces an embedding of $\mc[\tzmat]$ as a subalgebra of $\mc[\mat]$. 

Denote by $U$ a maximal unipotent subgroup of $\mat_1$ normalized by $T$. 
The algebra of covariants on $\mat$ is known classically from the theory of covariants of binary forms. 
The result in our notation can be stated as follows: 

\begin{proposition}\label{prop:formcovariants} 
(i) The algebra $\mc[\mat]^U$ is a polynomial ring generated by $\inv_1$ over $\mc[\tzmat]^U$. 

(ii) The algebra $\mc[\tzmat]^U$ is generated by $\inv_2$, $\inv_3$, $\tange_1$, $\tange_2$, $\tange_3$. 
\end{proposition} 

\begin{proof} 
Recall the well-known isomorphism $SO(3,\mc)\cong SL(2,\mc)/\{\pm I_2\}$, so $G$-modules  can be thought of as representations of  the special linear group 
$SL(2,\mc)$ with $-I_2$ in the kernel. 
This way 
the conjugation action of $G$ on $\tzmat$  can be  identified with 
the natural $SL(2,\mc)$-representation on the space of binary quartic forms.  
Generators (and relations) for the algebra of covariants of binary quartics were determined in nineteenth century invariant theory (see e.g. \cite{grosshans} or \cite{olver}).    
There are two algebraically independent invariants, one of degree $2$ and $3$. 
The covariant $\tange_2$ corresponds to the \emph{Hessian} 
covariant $\mathrm{Hess}$ mapping the binary quartic $Q=\sum_{i=0}^4a_ix^iy^{4-i}$ to the binary quartic 
$\mathrm{Hess}(Q):=\det(\left(\begin{array}{cc}\partial_{xx}Q & \partial_{xy}Q \\\partial_{yx}Q & \partial_{yy} Q\end{array}\right))$. 
The covariant  $\tange_3$ corresponds to the  map  sending the binary quartic 
$Q$ to the Jacobian of $Q$ and its Hessian, which is the binary sextic  
$\mathrm{Jac}(Q,\mathrm{Hess}(Q)):=\det(\left(\begin{array}{cc}\partial_{x}Q & \partial_{x}\mathrm{Hess}(Q) \\\partial_{y}Q & \partial_{y}\mathrm{Hess}(Q)\end{array}\right))$. 
\end{proof}

Similarly to Section~\ref{sec:tange},  write $\overline{\inv_i}$, $\overline{\tange_j}$ for the restriction to $\mat_1$ of $\inv_i$, $\tange_j$. Since $\mat$, $\mat_1$, $\tzmat$, $\tzmat_1$ are all subsets of the set denoted by the same symbol in Section~\ref{sec:tange}, and $\overline{\inv_i}$ and $\overline{\tange_j}$ are restrictions of the corresponding functions from Section~\ref{sec:tange}, as a corollary 
of Proposition~\ref{prop:firstrelations} we obtain that exactly the same relations hold with the  new scenario. 
Moreover, the statement of Theorem~\ref{thm:covsonF} remains valid, with verbatim the same proof. 

\begin{corollary} \label{cor:so3stable} 
As an $SO(3,\mc)$-stable ideal, $\ideal(\mat_1)$ is generated by $\tange_3^U$. 
\end{corollary} 
 
 \begin{proof} We apply Lemma~\ref{lemma:logic}: by Proposition~\ref{prop:formcovariants} and the new versions of  Proposition~\ref{prop:firstrelations} and 
Theorem \ref{thm:covsonF} discussed in the above paragraph we conclude that  $\ideal(\mat_1)$ is generated as a $G$-stable ideal by $\tange_3^U$, 
 $\inv_2^3-\inv_3^2$, $\inv_2\tange_2^U-\inv_3\tange_1^U$, $(\tange_2^U)^2-\inv_2(\tange_1^U)^2$, 
$\inv_3\tange_2^U-\inv_2^2\tange_1^U$. 
We know already from  Corollary~\ref{cor:sl3stable} that the elements of $\mc[\mc^{3\times 3}]$ denoted by the same symbols as the  latter four elements are contained in the ideal generated by 
the coordinate functions of $\tange_3$ (defined on $\mc^{3\times 3}$). Applying the natural surjection $\mc[\mc^{3\times 3}]\to \mc[\mat]$ given by restriction of functions to $\mat\subset \mc^{3\times 3}$ we conclude that the ideal generated by the coordinate functions of $\tange_3$ (interpreted as a covariant on $\mat$) contain the latter four elements of $\mc[\mat]$. So our statement follows, since the coordinate functions of $\tange_3$ span the $SO(3,\mc)$-module generated by $\tange_3^U$. 
 \end{proof}

\begin{corollary}\label{cor:main3}
(i) The ideal $\ideal(\mat_1)$ is generated by its degree $3$  component $\ideal(\mat_1)_3$. 

(ii) The $7$  coordinate functions of $\tange_3:\mat\to V^{(3)}$ constitute a $\mc$-basis in $\ideal(\mat_1)_3$.  

(iii) As an $SO(3,\mc)$-module, $\ideal(\mat_1)_3\cong V^{(3)}$, the space of $3$-variable spherical harmonics of degree $3$. 

(iv)  The coordinate functions of $\mathcal{P}_1:\mat\to\bigwedge^3\mat$ span $\ideal(\mat_1)_3$. 
\end{corollary}
 
 \begin{proof} Since $\tange_3^U$ is homogeneous of degree $3$, it follows trivially from Corollary~\ref{cor:so3stable} that $\ideal(\mat_1)$ is generated by its degree three homogeneous component, so (i)  follows. 
The covariant $\tange_3$ is non-zero and  irreducible, hence its coordinate  functions are linearly independent, so both (ii) and (iii) hold by Corollary~\ref{cor:so3stable} and by construction of $\tange_3$. 
Finally, (iv) follows from (ii) and Proposition ~\ref{prop:coordfunctions}.  
\end{proof} 

The {\it multiplicity series} of $\mc[\mat_1]$ is 
\[M(\mc[\mat_1];q,t)=\sum_{d=0}^\infty \sum_{n=0}^\infty m(d,n)q^nt^d\] 
where $m(d,n)$ denotes the multiplicity of the irreducible $SO(3,\mc)$-module $V^{(n)}$ as a summand in $\mc[\mat_1]_d$.  
The present variant of Theorem~\ref{thm:covsonF} yields 
\begin{equation}\label{eq:multiplicityseries} 
M(\mc[\mat_1];q,t)=\frac{1}{1-t}\left(\frac{1}{(1-t)(1-q^2t)}-t\right)
\end{equation}
The trace of $q\in T\in \mc^\times$ acting on $V^{(n)}$ is $\frac{q^{1/2}q^n-q^{-1/2}q^{-n}}{q^{1/2}-q^{-1/2}}$, hence the Hilbert series of 
$\mc[\mat_1]$ bigraded by $\mn_0\times \mathrm{Char}(T)$ is 
\[H(\mc[\mat_1];q,t)=\frac{q^{1/2}M(\mc[\mat_1];q,t)-q^{-1/2}M(\mc[\mat_1];q^{-1},t)}{q^{1/2}-q^{-1/2}}\] 

\begin{corollary}\label{cor:hilbseriessym}
The Hilbert series $H(\mc[\mat_1];t):= \sum_{d=0}^\infty \dim_{\mc}(\mc[\mat_1]_d)t^d$ equals 
\[H(\mr[\symmat(3,\mr)_1];t)=H(\mc[\mat_1];t)=\frac{1+2t+3t^2-3t^3+t^4}{(1-t)^4}.\]  
\end{corollary}

Finally we point out a connection between Corollary~\ref{cor:main3} and coincident root loci. 
Denote by $\mathrm{Pol}_d(\mc^2)$ the $SL(2,\mc)$-module of binary forms of degree $d$. 
Up to non-zero scalar multiples there is a unique $SL(2,\mc)$-module isomorphism 
$\varphi:\tzmat\to \mathrm{Pol}_4(\mc^2)$ (where we view the $SO(3,\mc)$-module $\tzmat$ an $SL(2,\mc)$-module via the surjection 
$SL(2,\mc)\to SO(3,\mc)$. 
As we pointed out in the proof of Proposition~\ref{prop:formcovariants}, the covariant $\tange_2$ corresponds to the \emph{Hessian} 
covariant $\mathrm{Hess}:\mathrm{Pol}_4(\mc)\to \mathrm{Pol}_4(\mc)$,  
and $\tange_3$ corresponds to the covariant $\mathrm{Pol}_4(\mc)\to \mathrm{Pol}_6(\mc)$ $Q\mapsto \mathrm{Jac}(Q,\mathrm{Hess}(Q))$. 
It is well known that the zero locus of the coefficient space of the latter covariant is the subset of binary quartics that are the square of a binary quadric 
(see  \cite{chipalkatti}), whence by Corollary~\ref{cor:main3} we conclude: 

\begin{proposition} \label{prop:quarticinterpretation} 
The $SL(2,\mc)$-equivariant vector space isomorphism $\varphi:\tzmat\to \mathrm{Pol}_4(\mc)$ maps the set $\tzmat_1$ of trace zero symmetric matrices with a minimal polynomial of degree at most $2$  onto the set of binary quartics that are the square of a binary quadric. 
\end{proposition} 

In fact it is  known that the coefficients of $\mathrm{Jac}(Q,\mathrm{Hess}(Q))$ generate the vanishing ideal of the set of binary quartics that are the square of a 
quadric, see   \cite{chipalkatti}, where this is stated (without the concrete computational details), after an explanation of  a general method for  
the study of ideals of coincident root loci in the space of binary forms of degree $d$.  So it would be  possible to derive our Corollary~\ref{cor:main3}  from this 
result with the aid of Proposition 4.1 in \cite{domokos} and Proposition~\ref{prop:coordfunctions} of the present paper.    
For further results on coincident root loci  see the papers 
\cite{chipalkatti},   \cite{paramanathan}, \cite{weyman3} (and the references therein). 

%%%%%%%%%%%%%%%%%%%%%%%%%%%%%%%

\section{Real forms and sums of squares}\label{sec:real}  

Recall that the compact real form $SU(n)$ is Zariski dense in the complex affine algebraic group  $SL(n,\mc)$, hence an irreducible $SL(n,\mc)$-module remains irreducible over $SU(n)$. 
For a compact real  Lie group $G$ and a finite dimensional complex $G$-module $V$ denote $V_{\mr}$ the {\it realification} of $V$, and for a finite dimensional 
real $G$-module $W$, its  {\it complexification} is $\mc\otimes_{\mr}W$.  
The realification $\mathrm{S}^n(\mc^n)^\star_{\mr}$ of the $n$th symmetric power of the dual of the natural $SU(n)$-module  $\mc^n$ is irreducible as a real representation of $SU(n)$, whereas  its complexification splits as 
\[\mc\otimes_{\mr} \mathrm{S}^n(\mc^n)^\star_{\mr}\cong \mathrm{S}^n(\mc^n)^\star \oplus \mathrm{S}^n(\mc^n)\] 
as a complex $SU(n)$-module. 
Set 
\begin{equation}\label{eq:tangeujra} \rootdisc:\hermat(n)\to\mathrm{S}^n(\mc^n)^\star_{\mr}, \quad  
A\mapsto (\underline{x} \mapsto \det (\underline{x}\vert A\underline{x}\vert \dots \vert A^{n-1}\underline{x})
\end{equation}
where for $\underline{x}\in \mc^n$ and $A\in\hermat(n)$ we write $(\underline{x}\vert A\underline{x}\vert\dots\vert A^{n-1}\underline{x}))$ for the  $n\times n$ matrix 
whose columns are $\underline{x}$, $A\underline{x}$, $\dots$, $A^{n-1}\underline{x}$,  and 
$\mathrm{S}^n(\mc^n)^\star$ is identified with the space of homogeneous forms of degree $n$ on $\mc^n$. 
For a diagonal matrix $A=\mathrm{diag}(a_1,\dots,a_n)$ we have 
\begin{equation} \label{eq:diagonal}
\rootdisc(A)(\underline{x})=x_1\dots x_n\prod_{1\le i<j\le n}(a_j-a_i)
\end{equation} 
hence $c$ is non-zero. 
We obtain the following statement: 

\begin{proposition}\label{prop:sunsubmodule}
The $2\binom{2n-1}{n-1}$ real coordinate functions of $\rootdisc$ span an $SU(n)$-submodule in $\ideal(\hermat(n)_1)_{\binom{n}{2}}$ isomorphic to 
$\mathrm{S}^n(\mc^n)^\star_{\mr}$. 
\end{proposition}

The same proof as for Theorem 4.1 in \cite{domokos2} yields the following: 

\begin{proposition}\label{prop:subdiscunique} 
Up to non-zero scalar multiples $\sdisc_k\in\mr[\hermat(n)]$ is the only $SU(n)$-invariant element in the degree $(n-k)(n-k-1)$ homogeneous component of 
$\ideal(\hermat(n)_{k+1})$, and there are no non-zero $SU(n)$-invariants in $\ideal(\hermat(n)_{k+1})$ of degree  less than  $(n-k)(n-k-1)$. 
\end{proposition} 

By   Lemma 2.1  in \cite{domokos} this yields: 

\begin{corollary}\label{cor:sumofsquares} 
The discriminant $\disc\in\mr[\hermat(n)]$ can be written as the sum of $2\binom{2n-1}{n-1}$ squares. 
\end{corollary}
 
 \begin{remark}\label{remark:gorodski} 
(i) The study of sum of squares representations of the discriminant of real symmetric matrices goes back to Kummer and Borchardt (see some references in \cite{domokos}, whose approach was inspired by \cite{lax:1998}). A relation to the entropic discriminant is established in \cite{sanyal-sturmfels-vinzant}. 
A sum of squares presentation of the discriminant of Hermitian matrices was shown by Newell \cite{newell}, Ilyushechkin \cite{ilyushechkin}, Parlett \cite{parlett}. Corollary~\ref{cor:sumofsquares} significantly reduces the number of summands in these presentations.  

(ii) Sum of squares presentations of discriminants for the isotropy representation of Riemannian symmetric spaces were studied by Gorodski \cite{gorodski} 
(and also in \cite{rais}). 
In particular, it is proved in \cite{gorodski} that  the discriminant associated to the symmetric space $Sp(n,\mr)/U(n)$ is the sum of $2\binom{2n-1}{n-1}$ squares (the corresponding representation of $U(n)$ is the  action $X\mapsto gXg^T$ on $\symmat(n,\mc)$). 
This number coincides with the number appearing in Corollary~\ref{cor:sumofsquares} above, but the associated symmetric space (and representation) is different, it is 
$SL(n,\mc)/SU(n)$ in our case. 

(iii) Similarly to \eqref{eq:tangeujra} consider the $SO(n)$-equivariant polynomial map
 \[ \rootdisc_{\mr}:\symmat(n,\mr)\to\mathrm{S}^n(\mr^n)^\star, \quad  
A\mapsto (\underline{x} \mapsto \det (\underline{x}\vert A\underline{x}\vert \dots \vert A^{n-1}\underline{x}))\] 
Since the $SO(n)$-orbit of any $A\in\symmat(n,\mr)$ contains a diagonal matrix and the Laplace operator  $\Delta:=\sum_{i=1}^n\frac{\partial ^2}{\partial  x_i^2}$ is 
$SO(n)$-equivariant, formula \eqref{eq:diagonal} shows that the  image of $\rootdisc_{\mr}$ is contained the space 
$\harmonics^n(\mr^n):=\mathrm{S}^n(\mr^n)^\star\cap \ker(\Delta)$ of $n$-variable spherical harmonics of degree $n$.  
Note that $SO(n)$-modules are self-dual. This shows that the $\mr$-subspace of the 
coordinate functions of $\rootdisc_{\mr}$ span an $SO(n)$-submodule in $\ideal(\symmat(n,\mr)_1)$ isomorphic to $\harmonics^n(\mr^n)$. 
Thus we obtained a more direct proof of the first statement of Theorem 6.2 from \cite{domokos} than the proof given in loc. cit..
 \end{remark}
 
 %%%%%%%%%%%%%%%%%%%%%%%%%%%%%%
 
 \section{Subdiscriminants of Hermitian matrices}\label{sec:subdischer} 
 
In this section we extend the results of \cite{domokos2} on real symmetric matrices to the case of Hermitian matrices. In particular, we generalize 
Proposition~\ref{prop:sunsubmodule} and Corollary~\ref{cor:sumofsquares} of the present paper for the $k$-subdiscriminant of Hermitian matrices with arbitrary  $k$. 
Let $U$ denote the subgroup of upper unitriangular matrices in $SL(n,\mc)$ acting by conjugation on $\mc^{n\times n}$, and $T$ the subgroup of diagonal matrices in $SL(n,\mc)$. We identify $\mz^{n-1}$ with the group of rational characters of $T$: $\lambda=(\lambda_1,\dots,\lambda_{n-1})\in\mz^{n-1}$ corresponds 
to $\mathrm{diag}(z_1,\dots,z_{n-1},(z_1\dots z_{n-1})^{-1})\mapsto z_1^{\lambda_1}\dots z_{n-1}^{\lambda_{n-1}}$.   
The irreducible $SL(n,\mc)$-modules are labeled by $\lambda\in\mz^{n-1}$ with $\lambda_1\ge\dots\ge \lambda_{n-1}\ge 0$. We denote by  $V^{\lambda}$ 
the irreducible $SL(n,\mc)$-module with highest weight $\lambda$. 

First we present a family of highest weight vectors in $\mc[\mc^{n\times n}]$ (introduced by Tange \cite{tange}) in the spirit of Section~\ref{sec:real}.  
For an $n\times n$ matrix $B$, $1\le i_1<\dots<i_s\le n$, and $1\le j_1<\dots <j_t\le n$ denote by $B_{i_1,\dots,i_s}^{j_1,\dots,j_t}$ the $s\times t$ submatrix of $B$ obtained by omitting the rows of index other than $i_1,\dots,i_s$ and the  columns   of index other than $j_1,\dots,j_t$. 
Define $\rootsdisc_k\in \mc[\mc^{n\times n}]$ ($k=1,\dots,n-1$) by 
\[\rootsdisc_k(A):=\det([Ae_1\vert A^2e_1\vert\dots\vert A^{n-k}e_1]_{k+1,k+2,\dots,n}^{1,\dots,n-k})\] 
where $e_1:=[1,0,\dots,0]^T$, $A\in\mc^{n\times n}$. For $g\in U$ we have $g^{-1}e_1=e_1$, so 
\begin{eqnarray*}(g^{-1}\cdot \rootsdisc_k)(A)=\rootsdisc_k(gAg^{-1})=
\det([gAe_1\vert gA^2e_1\vert\dots\vert gA^{n-k}e_1]_{k+1,\dots,n}^{1,\dots,n-k})\\
=\det(g_{k+1,\dots,n}^{k+1,\dots,n}[ Ae_1\vert A^2e_1\vert\dots\vert A^{n-k}e_1]_{k+1,\dots,n}^{1,\dots,n-k}=
\rootsdisc_k(A)
\end{eqnarray*}
by multiplicativity of the determinant and since $g_{k+1,\dots,n}^{k+1,\dots,n}$ is upper unitriangular, hence has determinant $1$. 
Thus $\rootsdisc_k$ is $U$-invariant. Moreover, we have 
\[\mathrm{diag}(z_1,\dots,z_{n-1},(z_1\dots z_{n-1})^{-1})\cdot \rootsdisc_k=z_1^{n-k+1}z_2z_3\dots z_k\rootsdisc_k \] 
So $\rootsdisc_k$ is a highest weight vector in $\mc[\mc^{n\times n}]$ of weight $(n-k+1,1^{k-1})$ 
(we write $1^r$ for the sequence $1,\dots,1$ with $r$ terms), therefore it generates an irreducible $SL(n,\mc)$-module isomorphic 
to $V^{(n-k+1,1^{k-1})}$. 
For an irreducible complex  $SL(n,\mc)$-module $V^{\lambda}$ write $V^{\lambda}_{\mr}$ for $V^{\lambda}$ viewed as a real representation of $SU(n)$. 

 We obtain the following extension of Proposition~\ref{prop:sunsubmodule} and  Corollary~\ref{cor:sumofsquares}, which correspond to the special case $k=0$: 
 
 \begin{theorem}\label{thm:main4} Let $n\ge 3$  and $0\le k\le n-3$ be integers. 
 \begin{itemize}
\item[(i)] For $d<\binom{n-k}2$ the degree $d$ homogeneous component of $\ideal(\hermat(n)_{k+1})\triangleleft \mr[\hermat(n)]$ is zero. 

\item[(ii)] The degree $\binom{n-k}{2}$ homogeneous component of  $\ideal(\hermat(n)_{k+1})$ contains an irreducible real $SU(n)$-submodule isomorphic to 
$V^{(n-k,1^{k})}_{\mr}$. 

\item[(iii)] The $k$-subdiscriminant $\sdisc_{k}\in\mr[\hermat(n)]$ can be written as the sum of $2\dim_{\mc}(V^{(n-k,1^k)})$ squares. 
\end{itemize}
 \end{theorem}  
 
 \begin{proof} (i) follows from Lemma 2.1 in \cite{domokos} and Proposition~\ref{prop:subdiscunique}, and the latter two statements together with (ii) imply (iii) as well. 
 
 In order to prove (ii) note that  for $A\in (\mc^{n\times n})_{k+1}$ the matrices 
 $I_n,A,\dots,A^{n-k-1}$ are linearly dependent, hence 
 \[\det([e_1\vert Ae_1\vert A^2e_1\vert\dots\vert A^{n-k-1}e_1]_{1,k+2,\dots,n}^{1,\dots,n-k})=0.\] 
 By elementary properties of the determinant the left hand side coincides with $\rootsdisc_{k+1}(A)$. 
 This shows that the degree $\binom{n-k}{2}$ highest weight vector  $\rootsdisc_{k+1}$ constructed above belongs to 
 $\mc\otimes_{\mr}\ideal(\hermat(n)_{k+1})=
 \ideal((\mc^{n\times n})_{k+1})$. 
 Thus the complexification of $\ideal(\hermat(n)_{k+1})_{\binom{n-k}{2}}$ contains the irreducible complex $SL(n,\mc)$-module $V^{(n-k,1^k)}$. View 
 $V^{(n-k,1^k)}$ as an irreducible complex $SU(n)$-module. 
 It is not self-conjugate, hence its realification $V^{(n-k,1^k)}_{\mr}$ is an irreducible real $SU(n)$-module. Consequently, $\ideal(\hermat(n)_{k+1})_{\binom{n-k}{2}}$ 
contains an $SU(n)$-submodule  $V^{(n-k,1^k)}_{\mr}$  (it is spanned by the real and imaginary parts of a $\mc$-basis of the $SU(n)$-module generated by $f_k$ in $\mc\otimes_{\mr}\mr[\hermat(n)]$). 
 \end{proof} 
 
 The Weyl dimension formula provides an explicit expression for  $\dim_{\mc}(V^{(n-k,1^k)})$, see for example page 303 in \cite{goodman-wallach}. 
 
 \begin{center} {\it Acknowledgement. } \end{center} 
We are grateful to J. K. Merikoski for asking in \cite{merikoski} whether the methods of \cite{domokos} work for Hermitian matrices, and to P. Frenkel, D. Jo\'o and M. Ra{\"\i}s 
for some discussions on the topic of this paper.

%%%%%%%%%%%%%%%%%%%%%%%%%%%%%%%%%%%%

\end{document}